\documentclass[11pt]{article}
\usepackage{amsmath,amsthm,amscd}
\usepackage{amssymb}
\usepackage{amsfonts}
\usepackage{stmaryrd}
\usepackage{hyperref}
\usepackage{enumerate}
\usepackage{mathrsfs}
\usepackage[usenames,dvipsnames]{xcolor}
\usepackage{comment}
\usepackage{authblk}
\usepackage{geometry}
\usepackage[all]{xy}
\usepackage{tikz}
\usetikzlibrary{arrows}

\theoremstyle{definition}
\newtheorem{dfn}{Definition}[section]
\newtheorem{example}[dfn]{Example}
\newtheorem{rem}[dfn]{Remark}
\theoremstyle{plain}
\newtheorem{thm}[dfn]{Theorem}
\newtheorem*{thm*}{Theorem}
\newtheorem*{thmA}{Theorem \ref{thm:main1}}
\newtheorem*{thmB}{Theorem \ref{thm:main2}}
\newtheorem*{thmC}{Theorem \ref{thm:main3}}
\newtheorem{prp}[dfn]{Proposition}
\newtheorem{cor}[dfn]{Corollary}
\newtheorem{lem}[dfn]{Lemma}
\newtheorem*{lem:factorization}{Lemma \ref{lem:factorization}}
\newtheorem*{lemC}{Lemma \ref{lem:main3-lemma}}

\title{Cylindrical Dyck paths and the Mazorchuk-Turowska equation}
\author[1]{Jonas T. Hartwig\thanks{\textit{Email:} \texttt{jth@iastate.edu}}}
\affil[1]{Department of Mathematics,
Iowa State University}
\author[2]{Daniele Rosso\thanks{\textit{Email:} \texttt{rosso@math.ucr.edu}}}
\affil[2]{Department of Mathematics,
University of California Riverside}
\date{ }

\newcommand{\al}{\alpha} \newcommand{\be}{\beta}
\newcommand{\ga}{\gamma} \newcommand{\Ga}{\Gamma}
 
\newcommand{\ep}{\varepsilon}

\newcommand{\la}{\lambda}
   \newcommand{\si}{\sigma}

\newcommand{\C}{\mathbb{C}}
\newcommand{\Z}{\mathbb{Z}}

\newcommand{\GCD}{\mathrm{GCD}}

\newcommand{\iv}[2]{\llbracket #1,#2 \rrbracket}

\renewcommand{\tilde}{\widetilde}
\newcommand{\un}{\underline}

\begin{document}
\maketitle
\begin{abstract}
We classify all solutions $(p,q)$ to the equation $p(u)q(u)=p(u+\beta)q(u+\alpha)$ where $p$ and $q$ are complex polynomials in one indeterminate $u$, and $\alpha$ and $\beta$ are fixed but arbitrary complex numbers.
This equation is a special case of a system of equations which ensures that certain algebras defined by generators and relations are non-trivial.
We first give a necessary condition for the existence of non-trivial solutions to the equation. Then, under this condition, we use combinatorics of generalized Dyck paths to describe all solutions and a canonical way to factor each solution into a product of irreducible solutions.
\end{abstract}

\section{Introduction}

The goal of this paper is to give a precise description of all solutions $(p,q)$ to the equation
\begin{equation}\label{eq:main-equation}
p(u)q(u)=p(u+\beta)q(u+\alpha)
\end{equation}
where $p$ and $q$ are complex polynomials in one indeterminate $u$,
and $\alpha$ and $\beta$ are fixed but arbitrary complex numbers.

\subsection{Background and motivation}

The motivation for this problem comes from the 
subject of \emph{twisted generalized Weyl algebras}.
We fix a commutative domain $R$, an $n$-tuple $\si=(\si_1,\ldots,\si_n)$ of commuting automorphisms of $R$ and an $n$-tuple $t=(t_1,\ldots,t_n)$ of non-zero elements from the center of $R$. To this data we associate the \emph{generalized Weyl algebra} (\emph{GWA}), denoted $R(\si,t)$ which is defined as the ring extension of $R$ by generators $X_1,Y_1,\ldots,X_n,Y_n$ modulo relations
\begin{subequations}\label{eq:GWA-relations}
\begin{gather}
\label{eq:GWA-relations1}
X_ir=\si_i(r)X_i,\quad Y_ir=\si_i^{-1}(r)Y_i,\quad
Y_iX_i=t_i,\quad X_iY_i=\si_i(t_i), \quad
X_iY_j=Y_jX_i,\;i\neq j,\\ 
\label{eq:GWA-relations2}
X_iX_j=X_jX_i,\quad Y_iY_j=Y_jY_i.
\end{gather}
\end{subequations}
The number $n$ is called the \emph{rank} (or \emph{degree}) of $R(\si,t)$.
These algebras were introduced in \cite{Bavula1992}, independently studied without a name in \cite{Jordan1993}, and under the name \emph{hyperbolic ring} in \cite{Rosenberg1995}. 
Examples in rank $1$ include the Weyl algebra $A_1(\C)$ and quantum Weyl algebra $A_1^q(\C)$; the enveloping algebras $U(\mathfrak{g})$ for $\mathfrak{g}=\mathfrak{sl}_2, \mathfrak{sl}_3^+$ where $\mathfrak{sl}_3^+$ is the $3$-dimensional Heisenberg Lie algebra of upper triangular $3\times 3$-matrices; quantizations $U_q(\mathfrak{g})$ for the above $\mathfrak{g}$; deformations of algebras of functions on a $2$-dimensional sphere \cite{Podles1987},\cite{BavBek2000}.
The class of generalized Weyl algebras is closed under taking tensor products \cite{Bavula1992} and taking invariants with respect to a graded automorphism of finite order \cite{FarSolSua2003}.
As observed in \cite{Bavula1992}, the canonical map $R\to R(\si,t)$ is injective if and only if the following holds:
\begin{equation}\label{eq:GWA-consistency}
\text{(GWA Consistency Equations)}\qquad
 \si_i(t_j)= t_j\quad \forall i\neq j.
\end{equation}
In particular, \eqref{eq:GWA-consistency} ensures that the algebra $R(\si,t)$ is non-trivial.

Dropping relations \eqref{eq:GWA-relations2} (and instead taking the quotient by a certain ideal, a detail irrelevant to the present paper) the resulting algebra is called a \emph{twisted generalized Weyl algebra} (\emph{TGWA})  and is denoted by $\mathscr{A}(R,\si,t)$. These algebras were introduced by Mazorchuk and Turowska in \cite{MazTur1999}. Interesting examples of of TGWAs include deformations of the higher degree Weyl algebras \cite{MazTur2002},\cite{FutHar2012b}; extended OGZ algebras and the Mickelsson-Zhelobenko step algebra $Z(\mathfrak{gl}_{n+1},\mathfrak{gl}_n\times \mathfrak{gl}_1)$ \cite{MazPonTur2003}; a family of algebras attached to multiquivers \cite{HarSer2015pre}; certain primitive quotients of $U(\mathfrak{gl}_n)$ and $U(\mathfrak{sp}_{2n})$ \cite{HarSer2015pre}. The representation theory for these algebras is particularly interesting \cite{MazTur1999},\cite{MazTur2002},\cite{MazPonTur2003},\cite{Hartwig2006},\cite{FutHar2012b}.
For TGWAs the consistency equations \eqref{eq:GWA-consistency} get weakened to the following:
\begin{equation}\label{eq:TGWA-consistency}
\text{(TGWA Consistency Equations)}
\qquad
\left\{\begin{aligned}
t_it_j &=\si_i^{-1}(t_j)\si_j^{-1}(t_i) &&\quad\forall i\neq j,\\ 
\si_i\si_k(t_j)t_j &=\si_i(t_j)\si_k(t_j)&&\quad \forall i\neq j\neq k\neq i.
\end{aligned}\right.
\end{equation}
More precisely, the canonical ring homomorphism $R\to\mathscr{A}(R,\si,t)$ is injective if and only if
\eqref{eq:TGWA-consistency} holds
 \cite[Theorem~A]{FutHar2012a}. That the first equation in \eqref{eq:TGWA-consistency} is necessary was also observed in \cite{MazTur1999}, therefore we refer to that equation as the \emph{Mazorchuk-Turowska equation}.
Thus, finding solutions $t_1,\ldots,t_n$ to \eqref{eq:TGWA-consistency} for a given ring $R$ and given automorphisms $\si_i$ is equivalent to obtaining non-trivial algebras $\mathscr{A}(R,\si,t)$.
Taking $n=2$,  \eqref{eq:TGWA-consistency} becomes just a single equation:
\begin{equation} \label{eq:Rank-2-TGWA-consistency}
\text{(Rank Two TGWA Consistency Equation)}\qquad 
t_1t_2=\si_2^{-1}(t_1)\si_1^{-1}(t_2).
\end{equation}
Specializing further to the case when $R$ is the algebra $\C[u]$ of complex polynomials in one indeterminate $u$ and the automorphisms $\si_1$ and $\si_2$ are given by additive shift automorphisms: 
$\si_1\big(f(u)\big)=f(u-\al)$, $\si_2\big(f(u)\big)=f(u-\be)$ for some complex numbers $\al,\be\in\C$, and putting $t_1=p(u)$ and $t_2=q(u)$, \eqref{eq:Rank-2-TGWA-consistency} turns into \eqref{eq:main-equation}.
The following result summarizes the above discussion.
\begin{thm*}[Special case of {\cite[Theorem~A]{FutHar2012a}}]
Let $\al$ and $\be$ be non-zero complex numbers and $p$ and $q$ be non-zero complex polynomials of one variable.
Let $\mathscr{C}(\alpha,\beta,p,q)$ be the unital associative algebra over $\C$ with
generators
\[U, X_1, X_2, Y_1, Y_2\]
and defining relations
\begin{align*}
X_1U&=(U-\al)X_1,& Y_1U&=(U+\al)Y_1, \\ 
X_2U&=(U-\be)X_2,& Y_2U&=(U+\be)Y_2, \\ 
Y_1X_1 &= p(U),& Y_2X_2&=q(U), \\ 
X_1Y_2 &= Y_2X_1,& X_2Y_1&=Y_1X_2.
\end{align*}
Then the following two statements are equivalent:
\begin{enumerate}[{\rm (i)}]
\item The element $U\in\mathscr{C}(\alpha,\beta,p,q)$ is algebraically independent over $\C$.
\item The polynomials $p$ and $q$ satisfy \eqref{eq:main-equation}.
\end{enumerate}
In particular, \eqref{eq:main-equation} is sufficient for the algebra $\mathscr{C}(\al,\be,p,q)$ to be non-trivial.
\end{thm*}

Some solutions to TGWA consistency equations have been found previously in \cite{MazTur1999},\cite{MazPonTur2003},\cite{Sergeev2001},\cite{Hartwig2010},\cite{HarSer2015pre}. In Section \ref{sec:previous-solutions} we discuss how the solutions from \cite{HarSer2015pre} are related to the present paper.

We expect that some of our methods can be applied to more general rings $R$ and possibly to higher rank. There is also a quantized version of the Mazorchuk-Turowska equation, introduced in \cite{MazTur2002}. In rank two this equation is $t_1t_2=\xi\si_2^{-1}(t_1)\si_1^{-1}(t_2)$ where $\xi$ is a fixed non-zero complex number. It would be interesting to classify solutions $(t_1,t_2)\in R\times R$ to that equation for some choice of ring $R$ and commuting automorphisms $\si_1$ and $\si_2$. Finally, we expect the representation theory for the algebras $\mathscr{C}(\al,\be,p,q)$ to be combinatorially interesting.
These will be subjects for future research.

\subsection{Description of the main results}
Our first main result gives a necessary condition on $\al$ and $\be$ for the existence of non-trivial solutions to \eqref{eq:main-equation}. Notice that we can rescale both $p$ and $q$ by non-zero constants in \eqref{eq:main-equation}. Hence it suffices to consider solutions where $p$ and $q$ are monic polynomials.
\begin{thmA}
Let $\al$ and $\be$ be non-zero complex numbers. Assume that $m\al+n\be\neq 0$ for all positive integers $m$ and $n$. Then the only solution $(p,q)$ to \eqref{eq:main-equation}, in which $p$ and $q$ are non-zero and monic, is the trivial solution $(p,q)=(1,1)$.
\end{thmA}

To state the second main result we need some terminology.

\begin{dfn}\label{dfn:lattice-path}A finite non-negative \emph{lattice path} with unit steps to the right and up is a sequence $\pi=\big((x_0,y_0),(x_1,y_1),\ldots,(x_N,y_N)\big)$ where $(x_i,y_i)\in (\Z_{\ge 0})^2$ and $(x_{i+1},y_{i+1})\in\{(x_i+1,y_i), (x_i,y_i+1)\}$ for all $i=0,1,\ldots,N-1$. The point $(x_0,y_0)$ is called the \emph{starting point} and $(x_N,y_N)$ the \emph{end point} of the path $\pi$. A \emph{generalized Dyck path} is a lattice path $\pi$ with starting point $(0,0)$ such that $my_i\ge nx_i$ for all $i$, where $(m,n)$ is the end point of $\pi$. 
We denote by $\mathscr{D}(m,n)$ the set of all generalized Dyck paths with endpoint $(m,n)$.
\end{dfn}
\begin{example} 
Here we draw two lattice paths as defined above, both with starting point $(0,0)$ and end point $(4,5)$. The green one is a generalized Dyck path, because it is always above the blue line $4y=5 x$, while the red one is not a generalized Dyck path.

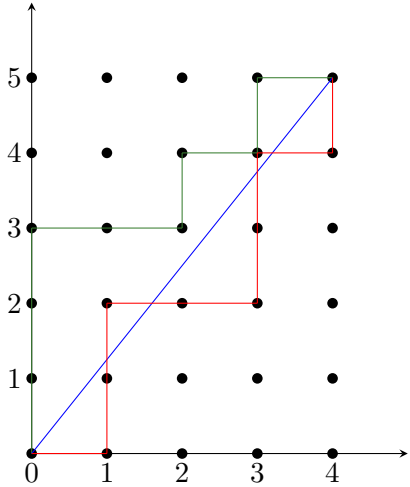
\vspace{.5cm}
\begin{center}
\begin{tikzpicture}[>=stealth,baseline={([yshift=-.5ex]current bounding box.center)}]
\draw[->] (0,0) to (5,0);
\draw[->] (0,0) to (0,6);
 \foreach \x in {0,1,2,3,4} {

\foreach \y in {0,1,2,3,4,5} {
\fill (\x,\y) circle(2pt); }}
 \foreach \x in {0,1,2,3,4} {
 \draw (\x,0) node [anchor=north] {$\x$};}
  \foreach \y in {1,2,3,4,5} {
 \draw (0,\y) node [anchor=east] {$\y$};}
 
 \draw[-] (0,0) to (0,3) [OliveGreen];
 \draw[-] (0,3) to (2,3) [OliveGreen];
 \draw[-] (2,3) to (2,4) [OliveGreen];
 \draw[-] (2,4) to (3,4)  [OliveGreen];
 \draw[-] (3,4) to (3,5) [OliveGreen];
 \draw[-] (3,5) to (4,5) [OliveGreen];
 \draw[-] (0,0) to (4,5) [blue];
 \draw[-] (0,0) to (1,0) [red];
 \draw[-] (1,0) to (1,2) [red];
 \draw[-] (1,2) to (3,2) [red];
 \draw[-] (3,2) to (3,4) [red];
 \draw[-] (3,4) to (4,4) [red];
 \draw[-] (4,4) to (4,5) [red];
\end{tikzpicture}
\end{center}
\vspace{.5cm}
\end{example}

Below and throughout this paper we will frequently use the standard product $(f,g)(p,q)=(fp,gq)$ in the direct product ring $\C[u]\times\C[u]$.
\begin{dfn}The solution $(p,q)=(1,1)$ is called the \emph{trivial} solution. A non-trivial solution $(p,q)$ to \eqref{eq:main-equation} in which $p$ and $q$ are non-zero and monic is called \emph{irreducible} if whenever $(p,q)=(p_1,q_1)(p_2,q_2)$, where $(p_i,q_i)\in\C[u]\times\C[u]$ solves \eqref{eq:main-equation} for $i\in\{1,2\}$, either $p_1$ and $q_1$ are constant or $p_2$ and $q_2$ are constant.
\end{dfn}
\begin{lem:factorization}
Any non-trivial solution $(p,q)$ to \eqref{eq:main-equation}, in which $p$ and $q$ are non-zero and monic, can be written 
\begin{equation} \label{eq:intro-lem-factorization}
(p,q)=(p_1,q_1)(p_2,q_2)\cdots (p_r,q_r)
\end{equation}
where $r\in\Z_{\ge 0}$ and each $(p_i,q_i)$ is an irreducible solution to \eqref{eq:main-equation}.
\end{lem:factorization}
The following theorem, which is the second main result of the paper, shows that the irreducible solutions of \eqref{eq:main-equation} are in bijective correspondence with the set $\mathscr{D}(m,n)\times\C$ for a certain pair $(m,n)$ uniquely determined by $(\al,\be)$.
\begin{thmB}
Let $\al$ and $\be$ be non-zero complex numbers. Assume that $m\al+n\be=0$ for some relatively prime positive integers $m$ and $n$. For any generalized Dyck path
\[\pi=\big((x_0,y_0),(x_1,y_1),\ldots,(x_{m+n},y_{m+n})\big)\in \mathscr{D}(m,n)\]
and any complex number $\la\in\C$, the pair of polynomials
\begin{equation}
\label{eq:intro-thm2-solutions} 
(p,q)=\left(P^{\pi,\la}_1,\, P^{\pi,\la}_2\right),
\quad
P_\ep^{\pi,\la}(u)=\prod_{i=1}^{m+n} \big(u-(\la+x_i\al+y_i\be)\big)^{\delta_{\ep,1}\delta_{x_i,x_{i-1}}+\delta_{\ep,2}\delta_{y_i,y_{i-1}}},\quad\ep\in\{1,2\},
\end{equation}
is an irreducible solution to \eqref{eq:main-equation} and,
conversely, any irreducible solution to \eqref{eq:main-equation} has the form \eqref{eq:intro-thm2-solutions} for some unique $(\pi,\la)\in \mathscr{D}(m,n)\times\C$.
\end{thmB}

At this point only one issue remains: The factorization \eqref{eq:intro-lem-factorization} into irreducible solutions is not unique (see Example \ref{ex:non-unique-factorization}). That is, the multiset $\{(p_1,q_1),\ldots,(p_r,q_r)\}$ is not uniquely determined by $(p,q)$. 
In the final part of the paper we resolve this issue.

The first step is to group irreducible factors together according to their zeroes modulo $\Z\al+\Z\be$.
\begin{dfn}\label{dfn:integral-sol}
A solution $(p,q)$ to \eqref{eq:main-equation}, in which $p$ and $q$ are non-zero and monic, is called \emph{integral} if every zero of $p(u)q(u)$ belongs to $\Z\al+\Z\be$. More generally, $(p,q)$ is called $c$-\emph{integral}, where $c\in\C/(\Z\al+\Z\be)$, if every zero of $p(u)q(u)$ belongs to the coset $c$.
\end{dfn}

\begin{lemC}
Let $\al$ and $\be$ be non-zero complex numbers.
For any solution $(p,q)$ to \eqref{eq:main-equation}, in which $p$ and $q$ are non-zero and monic, there exists a unique finite subset $C\subseteq\C/(\Z\al+\Z\be)$, and unique $c$-integral solutions $(p_c,q_c)\neq (1,1)$ for $c\in C$ such that
\begin{equation}
(p,q)=\prod_{c\in C} (p_c,q_c).
\end{equation}
\end{lemC}

Then we show that for each coset $c\in \C/(\Z\al+\Z\be)$, the set $\mathscr{D}(m,n)\times c$ can be equipped with a natural partial order $\preceq$ (see 
Section \ref{sec:ordered-factorizations} for details). 
The third and final main theorem of this paper shows that factorization of $c$-integral solutions into irreducibles is unique if one requires such factorizations to be ordered with respect to $\preceq$. It completes the classification of all solutions to \eqref{eq:main-equation}.

\begin{thmC}
Let $\al$ and $\be$ be non-zero complex numbers. Assume that $m\al+n\be=0$ for some relatively prime positive integers $m$ and $n$. Let $c\in \C/(\Z\al+\Z\be)$ and let $(p,q)$ be a $c$-integral solution to \eqref{eq:main-equation}. Then there exist a unique $k\in\Z_{\ge 0}$ and a unique sequence $\big((\pi_1,\la_1),(\pi_2,\la_2),\ldots,(\pi_k,\la_k)\big)$ of elements of $\mathscr{D}(m,n)\times c$ with
\begin{equation}
(\pi_1,\la_1)\preceq (\pi_2,\la_2)\preceq\cdots\preceq (\pi_k,\la_k)
\end{equation}
such that
\begin{equation}
(p,q)=
\left(P^{\pi_1,\la_1}_1,\, P^{\pi_1,\la_1}_2 \right)
\left(P^{\pi_2,\la_2}_1,\, P^{\pi_2,\la_2}_2 \right)
 \cdots
\left(P^{\pi_k,\la_k}_1,\, P^{\pi_k,\la_k}_2 \right),
\end{equation}
where the irreducible solutions $(P^{\pi,\la}_1,P^{\pi,\la}_2)$ for $(\pi,\la)\in \mathscr{D}(m,n)\times \C$ were defined in \eqref{eq:intro-thm2-solutions}.
\end{thmC}

\subsection*{Notation}
For $a,b\in\Z$ the set of integers $x$ satisfying $a\le x\le b$ is denoted by $\iv{a}{b}$. The set of non-zero complex numbers is denoted by $\C^\times$.

\section{Basic properties of solutions}

Let $R=\C[u]$ be the algebra of complex polynomials in one indeterminate $u$ and fix $\al_1, \al_2\in \C$. 
Our goal is to find solutions $(p,q)$ to the equation
\[p(u)q(u)=p(u+\al_2)q(u+\al_1)\]
Note that if $p=0$ or $q=0$ we trivially get a solution.
Hence we exclude those in the following definition.

\begin{dfn} For any $\al_1,\al_2\in\C$ we define the set
\begin{align*}
 S(\alpha_1,\alpha_2)
 =\big\{(p_1,p_2)\in (R\setminus\{0\})^2 &\mid \text{$p_1$, $p_2$ monic and } \text{$p_1(u)p_2(u)=p_1(u+\alpha_2)p_2(u+\alpha_1)$}\big\}.
\end{align*}
\end{dfn}

\begin{prp} Let $R_1$ be the set of monic polynomials in $R\setminus\{0\}$. Then $S(0,0)=R_1\times R_1$, and for any $\al\in\C^\times$, 
$S(0,\al)=\{1\}\times R_1$ and $S(\al,0)=R_1\times\{1\}$.
\end{prp}
\begin{proof} Obvious.
\end{proof}

In the rest of the paper we will assume that $\al_1$ and $\al_2$ are non-zero complex numbers.

\begin{prp}\label{prp:elementary-properties}
Let $\al_1,\al_2\in \C^\times$.
\begin{enumerate}[{\rm (a)}]
\item $S(\al_1,\al_2)$ is a multiplicative submonoid of the ring $R\times R$.
\item\label{part:divisibility-property}
$S(\al_1,\al_2)$ satisfies the following divisibility property: If $(p_1,p_2)\in S(\al_1,\al_2)$ and $(q_1,q_2)\in R\times R$ are such that $(p_1,p_2)(q_1,q_2)\in S(\al_1,\al_2)$ then $(q_1,q_2)\in S(\al_1,\al_2)$.
\item For each $\ga\in \C^\times$, there is a bijection $S(\al_1,\al_2)\to S(\al_1/\ga,\al_2/\ga)$ given by
\begin{equation}\label{eq:rescaling-bijection}
\big(p_1(u),p_2(u)\big)\mapsto \big(\ga^{-d_1}p_1(\ga u), \ga^{-d_2}p_2(\ga u)\big),
\end{equation}
where $d_i=\deg(p_i)$.
\end{enumerate}
\end{prp}
\begin{proof}

(a) The multiplicative unit $(1,1)$ of $R\times R$ belongs to $S(\al_1,\al_2)$. If $(p,q), (p',q')\in S(\al_1,\al_2)$ then $pp'$ and $qq'$ are monic and 
\begin{align*}
p(u)p'(u)q(u)q'(u)&=p(u)q(u)p'(u)q'(u) \\
&=p(u+\al_2)q(u+\al_1)p'(u+\al_2)q'(u+\al_1)\\ 
&=p(u+\al_2)p'(u+\al_2)q(u+\al_1)q'(u+\al_1),
\end{align*}
hence $(pp',qq')\in S(\al_1,\al_2)$. Since $(p,q)(p',q')=(pp',qq')$ by definition in the ring $R\times R$, this shows that $S(\al_1,\al_2)$ is a multiplicatively closed subset of $R\times R$.

(b) Straightforward to verify, using that $R$ is a domain.

(c) Suppose $(p_1,p_2)\in S(\al_1,\al_2)$. Then
\begin{align*}
\ga^{-d_1}p_1(\ga u)\ga^{-d_2}p_2(\ga u) &=
\ga^{-d_1}p_1(\ga u+\al_2)\ga^{-d_2}p_2(\ga u+\al_1) \\
&=\ga^{-d_1}p_1\left(\ga\left(u+\textstyle\frac{\al_2}{\ga}\right)\right)\ga^{-d_2}p_2\left(\ga\left(u+\textstyle\frac{\al_1}{\ga}\right)\right)
\end{align*}
which shows that $\big(\ga^{-d_1}p(\ga u), \ga^{-d_2} p(\ga u)\big)\in S(\al_1/\ga,\al_2/\ga)$.
The inverse is obtained by replacing $\ga$ by $\ga^{-1}$ in
\eqref{eq:rescaling-bijection}.
\end{proof}

We can now prove that any non-trivial solution in $S(\al_1,\al_2)$ can be written as a product of irreducible solutions, as stated in the introduction.
\begin{lem}\label{lem:factorization}
Any non trivial solution $(p,q)$ to \eqref{eq:main-equation}, in which $p$ and $q$ are non-zero and monic, can be written as
\begin{equation} \label{eq:lem-factorization}
(p,q)=(p_1,q_1)(p_2,q_2)\cdots (p_r,q_r)
\end{equation}
where $r\in\Z_{\ge 0}$ and each $(p_i,q_i)$ is an irreducible solution to \eqref{eq:main-equation}.
\end{lem}
\begin{proof}
Let $(p,q)\in S(\al_1,\al_2)$. If $(p,q)$ is irreducible we are done. If not, it can be factored as $(p',q')\cdot(p'',q'')$ where $(p',q'),(p'',q'')\in S(\al_1,\al_2)$ and $0<\deg(p'q'),\deg(p''q'')<\deg(pq)$. Repeat this argument for each of the factors $(p',q')$ and $(p'',q'')$. After finitely many (bounded by $\deg(pq)$) steps the process will terminate, at which point we have obtained the required factorization.
\end{proof}

\section{The free monoid on two generators}

In this section we introduce some notation that will be used in later sections and give a description of certain binary sequences.

Let $\mathsf{M}$ be the free monoid on two generators, regarded as the set of finite sequences $\underline{i}=(i_1i_2\cdots i_N)$ where $i_j\in \{1,2\}$ for all $j$. The product of two sequences is concatenation, denoted by juxtaposition. The \emph{length} of $\un{i}=(i_1i_2\cdots i_N)$, denoted $\ell(\un{i})$, is defined to be $N$. The unique sequence of length $0$ is denoted by $\emptyset$.
For $\un{i}=(i_1i_2\cdots i_N)\in\mathsf{M}$ we let
\begin{equation}\label{eq:ellm}\ell_m(\un{i})=\#\big\{k\in\iv{1}{N}\mid i_k=m\big\},\qquad m=1,2,\end{equation}
denote the number of occurrences of $m$ in the sequence $\un{i}$.
The infinite cyclic group $\Z$ acts on $\mathsf{M}$ by shifts: $1\cdot (i_1i_2\cdots i_N)=(i_2i_3\cdots i_Ni_1)$. The $\Z$-orbit in $\mathsf{M}$ containing $\un{i}$ is denoted by $\Z\cdot \un{i}$.
For each $\al_1,\al_2\in \C^\times$ define a submonoid $\mathsf{M}(\al_1,\al_2)$ of $\mathsf{M}$ as follows:
\begin{equation}\label{eq:Mal1al2-def}
\mathsf{M}(\al_1,\al_2)=\left\{\un{i}=(i_1i_2\cdots i_N)\in \mathsf{M}~\Big|~\textstyle \sum\limits_{k=1}^N \al_{i_k} =0 \right\} .
\end{equation}
We say that $\mathsf{M}(\al_1,\al_2)$ is \emph{trivial} if $\mathsf{M}(\al_1,\al_2)=\{\emptyset\}$.
Note that each submonoid $\mathsf{M}(\al_1,\al_2)$ is invariant under the $\Z$-action on $\mathsf{M}$.

\begin{dfn}
A sequence $\un{i}\in\mathsf{M}(\al_1,\al_2)$ 
 is called \emph{cyclically reducible} if there exist $n\in\Z$ and $\un{j},\un{k}\in\mathsf{M}(\al_1,\al_2)$ with $\un{j},\un{k}\neq\emptyset$ such that \[n\cdot\un{i}=\un{j}\un{k}.\]
A sequence $\un{i}\in\mathsf{M}(\al_1,\al_2)$ with $\un{i}\neq\emptyset$ which is not cyclically reducible is called \emph{cyclically irreducible}. Note that $\emptyset$ is not considered cyclically irreducible.
We let 
\[\mathsf{Irr}(\al_1,\al_2)=\left\{ \un{i}\in\mathsf{M}(\al_1,\al_2)~|~\un{i}\text{ is cyclically irreducible}\right\}.\]
Note that by definition $\mathsf{Irr}(\al_1,\al_2)$ is invariant under the $\Z$-action on $\mathsf{M}(\al_1,\al_2)$.
\end{dfn}

\begin{lem}\label{lem:al1al2}
Let $\al_1,\al_2\in \C^\times$. Then $\mathsf{M}(\al_1,\al_2)$ is non-trivial if and only if there exist relatively prime positive integers $m$ and $n$ such that 
\begin{equation}\label{eq:z-depend}m\al_1+n\al_2=0.\end{equation}
If this holds, then the pair $(m,n)$ is uniquely determined by $(\al_1,\al_2)$, and furthermore, for all $\un{i}\in\mathsf{M}(\al_1,\al_2)$ there exists a non-negative integer $d$ such that
\begin{align*}
\ell_1(\un{i})&=d\cdot m,\\
\ell_2(\un{i})&=d\cdot n.
\end{align*}
In particular, $m+n$ divides the length of every element of $\mathsf{M}(\al_1,\al_2)$.
\end{lem}
\begin{proof} Suppose there exists $\un{i}\in\mathsf{M}(\al_1,\al_2)$ with $\un{i}\neq\emptyset$. Let $N=\ell(\un{i})$. Then $\sum_{k=1}^N \al_{i_k}=0$, hence $m'\al_1+n'\al_2=0$ where $m'=\ell_1(\un{i})$ and $n'=\ell_2(\un{i})$. 
In this case, it is clear that $m=m'/\GCD(m',n')$ and $n= n'/\GCD(m',n')$ satisfy the required conditions. Uniqueness follows from the fact that $(m,n)=(|\al_2/\ga|,|\al_1/\ga|)$ where $\ga$ is a generator of the infinite cyclic group $\Z\al_1+\Z\al_2$.
Conversely, assume that $c_1\al_1+c_2\al_2=0$ for some positive integers $c_1,c_2$. 
Then the sequence $\un{i}=(11\cdots 122\cdots 2)$, where we have $c_k$ occurrences of $k$ for $k=1,2$, belongs to $\mathsf{M}(\al_1,\al_2)$.
\end{proof}

\begin{dfn}
Let $\underline{i}=(i_1i_2\cdots i_N)\in\mathsf{M}$.
An \emph{initial subsequence} of $\un{i}$ is a sequence in $\mathsf{M}$ of the form $(i_1i_2\cdots i_r)$ for some $r\in\iv{0}{N}$.
A \emph{cyclically consecutive subsequence} of $\underline{i}$ is an initial subsequence of $n\cdot \un{i}$ for some $n\in\Z$.
\end{dfn}

To describe all cyclically irreducible elements of $\mathsf{M}(\al_1,\al_2)$ we need the following lemma.

\begin{lem}\label{lem:subseq}
Let $\underline{i}=(i_1i_2\cdots i_N)\in\mathsf{M}$ and let $l_m=\ell_m(\un{i})$, as defined in \eqref{eq:ellm}, for $m=1,2$. We define $d=\GCD(l_1,l_2)$ and $s_m=\frac{l_m}{d}$ for $m=1,2$. Then there exists a cyclically consecutive subsequence $\underline{i}'=(i_n,i_{n+1},\ldots,i_{n+s_1+s_2-1})$ of $\underline{i}$ with the property that $\ell_1(\underline{i}')=s_1$ (and consequently $\ell_2(\un{i}')=s_2$). 
\end{lem}
\begin{proof}Let $S=\iv{1}{s_1+s_2}\subset \iv{1}{N}$ and consider the initial subsequence $\un{i}_S=(i_1i_2\cdots i_{s_1+s_2})\subset \un{i}$. For each $n\in\Z$, we let $\un{i}_S^n$ be the initial subsequence of length $s_1+s_2$ of $n\cdot\un{i}$. Notice that for each $n\in\Z$, the quantities $\ell_1(\un{i}_S^n)$ and $\ell_1(\un{i}_S^{n+1})$ differ at most by one. If $p:=\ell_1(\un{i}_S)=s_1$, we let $\underline{i}'=\underline{i}_S$ and the statement is proved. Now suppose that $p> s_1$. Consider the concatenated sequence $\un{I}=\un{i}_S(\un{i}^1_S)\cdots(\un{i}^{N-1}_S)$.  Since each index $j\in\iv{1}{N}$ appears exactly $s_1+s_2$ times in $\underline{I}$, we have $\ell_1(\underline{I})=(s_1+s_2)l_1$.
Hence 
$$\ell_1(\un{i}_S)+\ell_1(\un{i}^1_S)+\cdots+\ell_1(\un{i}^{N-1}_S)=(s_1+s_2)l_1=(s_1+s_2)ds_1=(l_1+l_2)s_1=Ns_1.$$

But since $\ell_1(\un{i}_S)=p$, with $p>s_1$, there exists $r\in\iv{0}{N-1}$ such that $\ell_1(\un{i}^r_S)=q$, with $q<s_1$. Since, as we observed, the quantity $\ell_1(\un{i}^n)$ can either be the same as $\ell_1(\un{i}^{n+1})$ or differ by one, it follows that for a choice of shift $0\leq r'\leq r$ we have necessarily that $\ell_1(\un{i}^{r'}_S)=s_1$, therefore $\underline{i}'=\underline{i}^{r'}_{S}$ is the cyclically consecutive subsequence we were looking for. If we had originally assumed that $p<s_1$, for the same reasons we would then have to have a $q>s_1$ and the conclusion would also follow.
\end{proof}

We can now prove the main result in this section.

\begin{prp}\label{prp:irred}Let $m\al_1+n\al_2=0$ for $m,n$ relatively prime positive integers and let $\underline{i}\in\mathsf{M}(\al_1,\al_2)$ be a non-trivial sequence. Then $\un{i}$ is cyclically irreducible if and only if $\ell(\un{i})=m+n$.
\end{prp}
\begin{proof}First of all, suppose that $\un{i}$ is cyclically reducible. For some $k\in\Z$ we have $k\cdot\un{i}=\un{i}'\un{i}''$ for two non trivial sequences $\un{i}',\un{i}''\in\mathsf{M}(\al_1,\al_2)$. By Lemma \ref{lem:al1al2}, $m+n$ divides both $\ell(\un{i}')$ and $\ell(\un{i}'')$, hence
$$\ell(\un{i})=\ell(k\cdot\un{i})=\ell(\un{i}')+\ell(\un{i}'')\geq 2(m+n).$$
Now suppose that $\un{i}\in\mathsf{M}(\al_1,\al_2)$ is such that $\ell(\un{i})>m+n$. By Lemma \ref{lem:al1al2}, there is an integer $d>1$ such that $\ell_2(\un{i})=d n$ and $\ell_1(\un{i})=d m$. But then Lemma \ref{lem:subseq} implies that there is a cyclically consecutive subsequence $\un{i}'$ of $\un{i}$ such that $\un{i}'\in\mathsf{M}(\al_1,\al_2)$ and $\ell(\un{i}')=m+n$. Write then $k\cdot \un{i}=\un{i}'\un{i}''$ for some other subsequence $\un{i}''$. By length considerations, $\un{i}''$ is non-trivial and, since $\un{i},\un{i}'\in\mathsf{M}(\al_1,\al_2)$, it follows that $\un{i}''\in\mathsf{M}(\al_1,\al_2)$ which shows  that $\un{i}$ is cyclically reducible.
\end{proof}

We can think of sequences in $\mathsf{M}$ as lattice paths in a very natural way. Let $\mathscr{P}_0$ be the set of all lattice paths as in Definition \ref{dfn:lattice-path} with starting point $(0,0)$.
\begin{dfn}\label{dfn:seq-to-path}We define a map $\mathsf{M}\to \mathscr{P}_0$ as follows. If $\un{i}=(i_1,\ldots,i_N)\in\mathsf{M}$, then $\un{i}\mapsto \pi(\un{i})\in\mathscr{P}_0$ where
$\pi(\un{i})=\left((x_0,y_0),(x_1,y_1),\ldots,(x_N,y_N)\right)$ is the path with $(x_0,y_0)=(0,0)$ and $(x_j,y_j)=(x_{j-1}+\delta_{1,i_j},y_{j-1}+\delta_{2,i_j})$ for all $j=1,\ldots, N$ defined recursively.
Conversely, we have a map going the other way $\mathscr{P}_0\to \mathsf{M}$, defined by $\pi\mapsto \un{i}(\pi)=(i_1,\ldots,i_N)$ where we set $i_j=1+\delta_{1,y_j-y_{j-1}}$ for all $j=1,\ldots,N$.
\end{dfn}
\begin{lem}The two maps of Definition \ref{dfn:seq-to-path} are inverses of each other, hence give a bijection between $\mathsf{M}$ and $\mathscr{P}_0$.
\end{lem}
\begin{proof}This is obvious.
\end{proof}
We can then define an action of $\Z$ on $\mathscr{P}_0$ by requiring it to commute with the bijection. That is,
\[ k\cdot \pi(\un{i})=\pi(k\cdot \un{i}) \qquad\forall~ k\in\Z,~\un{i}\in\mathsf{M}. \]
Under the bijection of Definition \ref{dfn:seq-to-path}, $\mathsf{M}(\al_1,\al_2)$ corresponds to the subset of $\mathscr{P}_0$ consisting of lattice paths with endpoints on the line $my=nx$, where $m,n$ are relatively prime positive integers with $m\al_1+n\al_2=0$. Moreover, $\mathsf{Irr}(\al_1,\al_2)$ corresponds to the subset of those where the endpoint is $(m,n)$. 
Recall that $\mathscr{D}(m,n)$ is the set of generalized Dyck paths (see Definition \ref{dfn:lattice-path}) with endpoint $(m,n)$.

\begin{lem}[{\cite[Lemma~2.3]{Fukukawa2013pre}}]\label{lem:fuku} Let $\al_1,\al_2\in\C^\times$ such that $m\al_1+n\al_2=0$ for relatively prime positive integers $m,n$. Then each $\Z$-orbit in $\mathsf{Irr}(\al_1,\al_2)$ contains exactly one sequence $\un{i}$ such that $\pi(\un{i})\in\mathscr{D}(m,n)$.
\end{lem}
Notice that Lemma \ref{lem:fuku} together with \cite[Theorem~2.4]{Fukukawa2013pre}, which gives a formula for the cardinality of $\mathscr{D}(m,n)$, can be used to count the number of $\Z$-orbits in $\mathsf{Irr}(\al_1,\al_2)$. We now give an independent proof of the same result.
\begin{prp} Let $\al_1,\al_2\in\C^\times$ and let $m\al_1+n\al_2=0$ for $m,n$ relatively prime positive integers. Then the number of $\Z$-orbits of cyclically irreducible sequences in $\mathsf{M}(\al_1,\al_2)$ is $\frac{1}{m+n}{m+n \choose m}$.
\end{prp}
\begin{proof}
We need to count $\# (Y/\Z)$, where $Y=\mathsf{Irr}(\al_1,\al_2)$. 
We know by Proposition \ref{prp:irred} that  $\un{i}\in\mathsf{Irr}(\al_1,\al_2)$ if and only if $\un{i}\in\mathsf{M}(\al_1,\al_2)$ and $\ell(\un{i})=m+n$.
Then the $\Z$-action on $\mathsf{Irr}(\al_1,\al_2)$ by shifts descends to an action of $\Z_{m+n}:=\Z/(m+n) \Z$. Since $\# (Y/\Z)=\# (Y/\Z_{m+n})$, we can then use Burnside's Lemma to count
\begin{equation}\label{eq:burnside} \#(Y/\Z_{m+n})=\frac{1}{m+n}\sum_{\bar{k}\in\Z_{m+n}}\#(Y^{\bar{k}})=\frac{1}{(m+n)}\sum_{\bar{k}=0}^{m+n-1}\#(Y^{\bar{k}}). \end{equation}
Now suppose that $0<k<m+n$ and that $\un{i}\in Y$ is a sequence such that $\bar{k}\cdot\un{i}=\un{i}$. Then we have necessarily $\un{i}=(i_1,i_2,\ldots,i_s,i_1,i_2,\ldots,i_s,\ldots,i_1,i_2,\ldots,i_s)$ for some $s$ such that $s|k$ and $s|(m+n)$.
It follows that $\ell_j(\un{i})=\frac{m+n}{s}\ell_j(\un{i}')$ for $j=1,2$, where $\un{i}'=(i_1,i_2,\ldots,i_s)$. But by Lemma \ref{lem:al1al2} we know that for all $\un{i}\in Y$ we have $\ell_1(\un{i})=m$ and $\ell_2(\un{i})=n$ which are relatively prime integers. This would imply that $s=m+n$, which is impossible since $k<m+n$, therefore $Y^{\bar{k}}=\emptyset$ for all $\bar{k}\neq\bar{0}$.
We then rewrite \eqref{eq:burnside} as
$$  \#(Y/\Z_{m+n})=\frac{1}{m+n}(\# Y^{\bar{0}})=\frac{1}{m+n}(\# Y).$$
To conclude, observe that to count all sequences in $Y$ we just have to chose which entries are $1$'s and which ones are $2$'s, and since for all $\un{i}\in \mathsf{Irr}(\al_1,\al_2)$ we have $\ell(\un{i})=m+n$ and $\ell_1(\un{i})=m$, there are exactly ${m+n \choose m}$ possibilities.
\end{proof}

\section{Fundamental solutions}

\begin{dfn} \label{dfn:fundamental-solutions}
Let $\al_1,\al_2\in \C^\times$.
To each $\un{i}\in \mathsf{M}(\al_1,\al_2)$ we attach a pair of monic polynomials $P^{\un{i}}=(P_1^{\un{i}},P_2^{\un{i}})\in R\times R$ given by
\begin{equation}\label{eq:Pi-def}
P_m^{\un{i}}(u)=(u-\eta_1)^{1-\delta_{i_1,m}}(u-\eta_2)^{1-\delta_{i_2,m}}\cdots (u-\eta_N)^{1-\delta_{i_N,m}},\qquad m=1,2,
\end{equation}
where $\eta_k =\eta_k(\un{i}) = \al_{i_1}+\al_{i_2}+\cdots +\al_{i_k}$ for $k\in\iv{1}{N}$ and $N=\ell(\un{i})$.
\end{dfn}

\begin{prp} \label{prp:fundamental-solution}
If $\al_1,\al_2\in \C^\times$ and $\un{i}\in \mathsf{M}(\al_1,\al_2)$ then $P^{\un{i}}\in S(\al_1,\al_2)$.
\end{prp}
\begin{proof} Let $N=\ell(\un{i})$. We have
\begin{equation}\label{eq:Pi-prp1}
P_1^{\un{i}}(u)P_2^{\un{i}}(u) = (u-\eta_1)(u-\eta_2)\cdots (u-\eta_N).
\end{equation}
On the other hand, we have
\[ (u+\al_2-\eta_k)^{1-\delta_{i_k,1}} = (u+\al_{i_k}-\eta_k)^{1-\delta_{i_k,1}}\] 
and 
\[ (u+\al_1-\eta_k)^{1-\delta_{i_k,2}} = (u+\al_{i_k}-\eta_k)^{1-\delta_{i_k,2}},\] 
which implies that 
\begin{equation}\label{eq:Pi-prp2}
P_1^{\un{i}}(u+\al_2)P_2^{\un{i}}(u+\al_1)
=(u+\al_{i_1}-\eta_1)(u+\al_{i_2}-\eta_2)\cdots (u+\al_{i_N}-\eta_N).
\end{equation}
But $\al_{i_k}-\eta_k = -\eta_{k-1}$ for $k>1$ and $\al_{i_1}-\eta_1 = 0 = -\eta_N$ because $\un{i}\in\mathsf{M}(\al_1,\al_2)$. This proves that the right hand sides of \eqref{eq:Pi-prp2} and \eqref{eq:Pi-prp1} are equal.
Thus $(P_1^{\un{i}},P_2^{\un{i}})\in S(\al_1,\al_2)$.
\end{proof}

The solutions $P^{\un{i}}$ for $\un{i}\in\mathsf{M}(\al_1,\al_2)$ are called the \emph{fundamental solutions} 
of $S(\al_1,\al_2)$.

\begin{example} 
Let $\al_1=2$, $\al_2=-3$ and let $\un{i}=(21211)$. Clearly $\un{i}\in\mathsf{M}(2,-3)$. We have $m2+n(-3)=0$ for $(m,n)=(3,2)$.
We then get 
\[ \eta_1=\al_2=-3;\qquad \eta_2=\al_2+\al_1=-1;\qquad \eta_3=\al_2+\al_1+\al_2=-4; \]
\[ \eta_4=\al_2+\al_1+\al_2+\al_1=-2; \qquad \eta_5=\al_2+\al_1+\al_2+\al_1+\al_1=0;\]
and therefore
\begin{align*}
P^{(21211)}_1(u) &= (u-\eta_1)(u-\eta_3)=(u+3)(u+4),\\ 
P^{(21211)}_2(u) &= (u-\eta_2)(u-\eta_4)(u-\eta_5)=u(u+1)(u+2).
\end{align*}
One can directly check that
\[P^{(21211)}_1(u+\al_2)P^{(21211)}_2(u+\al_1)=P^{(21211)}_1P^{(21211)}_2.\]
\end{example}
\begin{example}
Let $\al_1=5$, $\al_2=-3$ and let $\un{i}=(11221222)$. Clearly $\un{i}\in\mathsf{M}(5,-3)$. We have $m5+n(-3)=0$ for $(m,n)=(3,5)$.
We then get
\[ \eta_1=5;\qquad \eta_2=10;\qquad \eta_3=7;\qquad \eta_4=4; \]
\[ \eta_5=9; \qquad \eta_6=6;\qquad \eta_7=3;\qquad \eta_8=0;\]
and therefore
\begin{align*}
P^{(11221222)}_1(u) &= (u-\eta_3)(u-\eta_4)(u-\eta_6)(u-\eta_7)(u-\eta_8)=(u-7)(u-6)(u-4)(u-3)u\\ 
P^{(11221222)}_2(u) &= (u-\eta_1)(u-\eta_2)(u-\eta_5)=(u-10)(u-9)(u-5)
\end{align*}
Again one can directly check that $(P^{(11221222)}_1,P^{(11221222)}_2)\in S(5,-3)$.
\end{example}

\begin{dfn} Two solutions $(p_1,p_2)\in R\times R$ and $(q_1,q_2)\in R\times R$ are \emph{equivalent}, written $(p_1,p_2)\sim (q_1,q_2)$, if there exists a complex number $b$ such that $p_1(u)=q_1(u-b)$ and $p_2(u)=q_2(u-b)$.
\end{dfn}

If $P=(p(u),q(u))\in R\times R$ and $\la\in\C$ we put $P(u-\la)=\big(p(u-\la),q(u-\la)\big)\in R\times R$.

\begin{prp} \label{prp:fundamental-properties}
Let $\al_1,\al_2\in \C^\times$. The fundamental solutions $P^{\un{i}}$ satisfy the following properties.
\begin{enumerate}[{\rm (i)}]
\item\label{part:i}
The map
\begin{align*}
P:\mathsf{M}(\al_1,\al_2)&\longrightarrow S(\al_1,\al_2)\\ 
\un{i} &\longmapsto P^{\un{i}}
\end{align*}
 is a homomorphism of multiplicative monoids. That is 
 \[ P^{\un{i} \un{j}} = P^{\un{i}}\cdot P^{\un{j}}\quad\text{and}\quad P^{\emptyset} = (1,1).\] 
\item\label{part:iii}
If $P^{\un{i}}\sim P^{\un{j}}$ for some $\un{i},\un{j}\in\mathsf{M}(\al_1,\al_2)$, then $\ell_m(\un{i})=\ell_m(\un{j})$ for $m=1,2$.
\item\label{part:zeroes-of-P}
If $\ell(\un{i})=N$, then the set of zeroes of $P_1^{\un{i}}(u)P_2^{\un{i}}(u)$ is equal to 
\[\big\{\al_{i_1}, \al_{i_1}+\al_{i_2}, \cdots , \al_{i_1}+\al_{i_2}+\cdots+\al_{i_N}\big\}\]
Note that  $\al_{i_1}+\al_{i_2}+\cdots+\al_{i_N}=0$ since $\un{i}\in\mathsf{M}(\al_1,\al_2)$.
\item\label{part:P-no-repeated-roots}
If $\un{i}$ is cyclically irreducible, then $P^{\un{i}}_1(u)P^{\un{i}}_2(u)$ has no repeated roots.
\item\label{part:iv}
For any $k\in\iv{0}{N-1}$, where $N=\ell(\un{i})$ we have
\begin{equation}\label{eq:fundamental-shifts}
P_m^{k\cdot \un{i}}(u) = P_m^{\un{i}}\big(u+(\al_{i_1}+\al_{i_2}+\cdots +\al_{i_k})\big),\quad m=1,2.
\end{equation}
In particular, if $\Z\cdot \un{i}=\Z\cdot \un{j}$ then $P^{\un{i}}\sim P^{\un{j}}$.
\item\label{part:equality-of-fundamentals}
Let $(\un{i},\la),(\un{j},\mu)\in \mathsf{Irr}(\al_1,\al_2)\times\C$. Then
$P^{\un{i}}(u-\la)=P^{\un{j}}(u-\mu)$ if and only if there exists an integer $k\in\iv{0}{N-1}$ where $N=\ell(\un{i})$
such that $\un{j}=k\cdot\un{i}$ and $\mu=\la+\al_{i_1}+\cdots+\al_{i_k}$.
\end{enumerate}
\end{prp}
\begin{proof}

\eqref{part:i}-\eqref{part:zeroes-of-P} Straightforward to verify.

\eqref{part:P-no-repeated-roots} Suppose that $\al_{i_1}+\cdots+\al_{i_r}=\al_{i_1}+\cdots+\al_{i_s}$ for some $r,s\in\iv{1}{N}$, $r<s$. Then $\al_{i_{r+1}}+\cdots+\al_{i_s}=0$ which means that $\un{i}'=(i_{r+1}i_{r+2}\cdots i_s)\in\mathsf{M}(\al_1,\al_2)$ which implies that $\un{i}$ is cyclically reducible.

\eqref{part:iv} First suppose that $k=1$, and put $\un{j}=1\cdot \un{i}$, so that 
  $(j_1j_2\cdots j_N)=(i_2i_3\cdots i_Ni_1)$. The factors of $P^{\un{j}}_m$ have the form
$\big(u-\eta_r(\un{j})\big)^{1-\delta_{m,j_r}}$
where $r\in\iv{1}{N}$. 
If $r<N$, we have
\begin{equation*}
\eta_r(\un{j})=\al_{j_1}+\cdots +\al_{j_r} = \al_{i_2} + \cdots + \al_{i_{r+1}} =\eta_{r+1}(\un{i})-\al_{i_1}. 
\end{equation*}
Also
\[ \eta_N(\un{j})=\al_{j_1}+\cdots+\al_{j_N}=0=\al_{i_1}-\al_{i_1}=\eta_1(\un{i})-\al_{i_1}.\]
This shows that
\[\big(u-\eta_r(\un{j})\big)^{1-\delta_{m,j_r}}
=
\big(u-\eta_{r+1}(\un{i})+\al_{i_1})^{1-\delta_{m,i_{r+1}}}.
\]
Taking the product over all $r$ in $\iv{1}{N}$ we obtain \eqref{eq:fundamental-shifts}.
The general case then follows by repeated applications of the case $k=1$. 

\eqref{part:equality-of-fundamentals} If $\un{j}=k\cdot \un{i}$ and $\mu=\la+\al_{i_1}+\cdots+\al_{i_k}$, then part \eqref{part:iv} directly implies that $P^{\un{i}}(u-\la)=P^{\un{j}}(u-\mu)$.
Conversely, suppose that $P^{\un{i}}(u-\la)=P^{\un{j}}(u-\mu)$. By part \eqref{part:zeroes-of-P}, $\mu$ is a zero of $P^{\un{j}}_1(u-\mu)P^{\un{j}}_2(u-\mu)$, hence $\mu$ is a zero of $P^{\un{i}}_1(u-\la)P^{\un{i}}_2(u-\la)$. By part \eqref{part:zeroes-of-P} again, 
\[\mu=\la+\al_{i_1}+\al_{i_2}+\cdots+\al_{i_k}\]
for some $k\in\iv{0}{N-1}$. Thus
\[P^{\un{i}}(u-\la)=P^{\un{i}}(u-\mu+\al_{i_1}+\al_{i_2}+\cdots+\al_{i_k})=P^{k\cdot\un{i}}(u-\mu),\]
where we used part \eqref{part:iv} in the last equality.
Thus it suffices to show that $P^{\un{i}}=P^{\un{j}}$ implies $\un{i}=\un{j}$. Suppose $P^{\un{i}}=P^{\un{j}}$. Then 
$P^{\un{i}}_m=P^{\un{j}}_m$ for $m=1,2$.
Recall that $\eta_r(\un{k})=\al_{k_1}+\cdots+\al_{k_r}$ for $\un{k}\in\mathsf{M}(\al_1,\al_2)$.
By definition of $P^{\un{k}}$, $\eta_r(\un{k})$ is a zero of $P^{\un{k}}_{3-k_r}$ for all $r\in\iv{1}{N}$. For $r=N$ we get that $\eta_N(\un{i})=0=\eta_N(\un{j})$ is a zero of $P^{\un{i}}_{3-i_N}$ and of $P^{\un{j}}_{3-j_N}$. On the other hand $P^{\un{i}}_{3-i_N}=P^{\un{j}}_{3-i_N}$. Since 
 $\un{i}$ and $\un{j}$ are cyclically irreducible, part \eqref{part:P-no-repeated-roots} implies that $3-i_N=3-j_N$ i.e. $i_N=j_N$.
Hence $\eta_{N-1}(\un{i})=\eta_{N-1}(\un{j})$, which is a common zero of $P^{\un{i}}_{3-i_{N-1}}$ and $P^{\un{j}}_{3-j_{N-1}}$. Again by part \eqref{part:P-no-repeated-roots}, $i_{N-1}=j_{N-1}$. Continuing like this we eventually get $\un{i}=\un{j}$.
\end{proof}
\begin{rem}\label{rem:unique-Dyck-sequence}
By Proposition \ref{prp:fundamental-properties}\eqref{part:equality-of-fundamentals}, the pair $(\un{i},\la)$ is not uniquely determined by the solution $P^{\un{i}}(u-\la)$. However, by Lemma \ref{lem:fuku}, there is a unique choice if we require that $\pi(\un{i})\in\mathscr{D}(m,n)$.
\end{rem}

As a corollary we obtain a necessary condition for a fundamental solution $P^{\un{i}}$ to be irreducible. 

\begin{cor} \label{cor:irreducible-fundamental-solution}
Let $\al_1,\al_2\in\C^\times$ and assume $\emptyset\neq\un{i}\in\mathsf{M}(\al_1,\al_2)$. Let $(m,n)$ be the pair of relatively prime positive integers such that $m\al_1+n\al_2=0$. Suppose that $P^{\un{i}}$ is irreducible. Then $\ell_1(\un{i})=m$ and $\ell_2(\un{i})=n$.
In particular, $\ell(\un{i})=m+n$.
\end{cor}
\begin{proof} Assume that $P^{\un{i}}$ is irreducible. Put $l_m=\ell_m(\un{i})$ for $m=1,2$.
By Lemma \ref{lem:al1al2}, $l_1=dm$ and $l_2=dn$ for some positive integer $d$. In fact $d=\GCD(l_1,l_2)$ since $\GCD(m,n)=1$.
By Lemma \ref{lem:subseq} there exists $n\in\Z$ such that $n\cdot\un{i}=\un{i}'\un{i}''$ where $\un{i}'$ and $\un{i}''$ belong to $\mathsf{M}(\al_1,\al_2)$, and where $\ell_1(\un{i}')=l_1/d=m$ and $\ell_2(\un{i}')=l_2/d=n$.
By Proposition \ref{prp:fundamental-properties}~\eqref{part:i} and \eqref{part:iv}, 
\[ P^{\un{i}}\sim P^{n\cdot \un{i}}=P^{\un{i}'\un{i}''}=P^{\un{i}'}P^{\un{i}''}. \] 
Since $P^{\un{i}}$ is irreducible, $P^{n\cdot \un{i}}$ is also irreducible. But since $P^{\un{i}'}$ is non-constant, it follows that $P^{\un{i}''}$ is constant, which implies that $\un{i}''=\emptyset$. Thus $P^{\un{i}}\sim P^{n\cdot\un{i}}=P^{\un{i}'}$. By Proposition \ref{prp:fundamental-properties}~\eqref{part:iii}, we get $\ell_m(\un{i})=\ell_m(\un{i}')$ for $m=1,2$. Thus $d=1$.
\end{proof}

\begin{prp} \label{prp:factorization}
Let $\al_1,\al_2\in \C^\times$. Then any non-trivial solution $Q\in S(\al_1,\al_2)$ is divisible by a shifted fundamental solution $P^{\un{i}}(u-\la)$ for some non-empty sequence $\un{i}\in\mathsf{M}(\al_1,\al_2)$ and some $\la\in\C$.
\end{prp}

\begin{proof}
Let $Q=(q_1,q_2)\in S(\al_1,\al_2)$, $Q\neq (1,1)$. Let $\la_0$ be any root of $q_1(u)q_2(u)$.
We recursively define an infinite sequence
 $j_0,j_1,j_2,\ldots$ of elements from $\{1,2\}$ satisfying
\[q_{3-j_k}(\la_k)=0\qquad\text{for all $k\in\Z_{\ge 0}$,}\]
or, equivalently,
\begin{equation}\label{eq:fk-property}
 (u-\la_k)^{1-\delta_{j_k,m}} \big| q_m(u) \qquad\text{for all $k\in\Z_{\ge 0}$ and $m\in\{1,2\}$,}
\end{equation}
where we put $\la_k=\la_0+\al_{j_1}+\al_{j_2}+\cdots+\al_{j_k}$ for $k\in\Z_{>0}$.
 Define $j_0\in\{1,2\}$ such that $q_{3-j_0}(\la_0)=0$.
For $k\in\Z_{\ge 0}$, assume $j_k\in\{1,2\}$ is such that $q_{3-j_k}(\la_k)=0$.
Then $\la_k$ is a zero of $q_1(u)q_2(u)=q_1(u+\al_2)q_2(u+\al_1)$. Define $j_{k+1}\in\{1,2\}$ such that $\la_k$ is a root of $q_{3-j_{k+1}}(u+\al_{j_{k+1}})$. Then $\la_{k+1}$ is a root of $q_{3-j_{k+1}}(u)$.

Since the set of roots of $q_1(u)q_2(u)$ is finite, there exists a smallest positive integer $M$ for which $\la_{M+r}=\la_r$ for some $r\in\Z_{\ge 0}$. By definition of $\la_k$ this means that
\begin{equation}\label{eq:sigma-product}
\la_r+\al_{j_{r+1}}+\al_{j_{r+2}}+\cdots+\al_{j_{r+M}}=\la_r,
\end{equation}
which implies that the sequence
 $\un{i}:=(j_{r+1}j_{r+2}\cdots j_{r+M})$ belongs to 
$\mathsf{M}(\al_1,\al_2)$. By minimality of $M$, the numbers $\la_{r+1},\ldots,\la_{r+M}$ are pairwise different.
Thus \eqref{eq:fk-property} implies that $P_m^{\un{i}}(u-\la_r)$ divides $q_m$ for $m=1,2$. Thus $Q$ is divisible by $P^{\un{i}}(u-\la_r)$.
\end{proof}

Now we can prove the first two main results stated in the introduction.

\begin{thm}\label{thm:main1} 
Let $\al$ and $\be$ be non-zero complex numbers. Assume that $m\al+n\be\neq 0$ for all positive integers $m$ and $n$. Then the only solution $(p,q)$ to \eqref{eq:main-equation}, in which $p$ and $q$ are non-zero and monic, is the trivial solution $(p,q)=(1,1)$.
\end{thm}
\begin{proof}
Suppose $(1,1)\neq Q\in S(\al,\be)$. By Proposition \ref{prp:factorization}, $Q$ is divisible by a shifted fundamental solution $P^{\un{i}}(u-\la)$ where $\emptyset\neq \un{i}\in \mathsf{M}(\al_1,\al_2)$ and $\la\in\C$. By Lemma \ref{lem:al1al2}, there exist positive integers $m$ and $n$ such that $m\al+n\be=0$.
\end{proof}
The following easily checked result expresses the fundamental solutions in terms of Dyck paths.
\begin{lem}\label{lem:fundamental-Dyck-solutions}
Let $(\pi,\la)\in\mathscr{D}(m,n)$ and $\un{i}=\un{i}(\pi)$. Define $P^{\pi,\la}=P^{\un{i}}(u-\la)$.
Then 
\[P^{\pi,\la}=\left(P^{\pi,\la}_1,\, P^{\pi,\la}_2\right),
\quad
P_\ep^{\pi,\la}(u)=\prod_{i=1}^{m+n} \big(u-(\la+x_i\al+y_i\be)\big)^{\delta_{\ep,1}\delta_{x_i,x_{i-1}}+\delta_{\ep,2}\delta_{y_i,y_{i-1}}},\quad\ep\in\{1,2\}. \]
\end{lem}
\begin{thm}\label{thm:main2}
Let $\al$ and $\be$ be non-zero complex numbers. Assume that $m\al+n\be=0$ for some relatively prime positive integers $m$ and $n$. For any generalized Dyck path
\[\pi=\big((x_0,y_0),(x_1,y_1),\ldots,(x_{m+n},y_{m+n})\big)\in \mathscr{D}(m,n)\]
and any complex number $\la\in\C$, the pair of polynomials
\begin{equation}
\label{eq:thm2-solutions} 
(p,q)=\left(P^{\pi,\la}_1,\, P^{\pi,\la}_2\right),
\quad
P_\ep^{\pi,\la}(u)=\prod_{i=1}^{m+n} \big(u-(\la+x_i\al+y_i\be)\big)^{\delta_{\ep,1}\delta_{x_i,x_{i-1}}+\delta_{\ep,2}\delta_{y_i,y_{i-1}}},\;\;\ep\in\{1,2\},
\end{equation}
is an irreducible solution to \eqref{eq:main-equation} and,
conversely, any irreducible solution to \eqref{eq:main-equation} has the form \eqref{eq:thm2-solutions} for some unique $(\pi,\la)\in \mathscr{D}(m,n)\times\C$.
\end{thm}
\begin{proof}
 Let $Q$ be an irreducible solution in $S(\al,\be)$. By Proposition \ref{prp:factorization}, $Q$ is equal to $P^{\un{i}}(u-\mu)$ for some $\un{i}\in\mathsf{M}(\al,\be)$ and $\mu\in\C$. By Corollary \ref{cor:irreducible-fundamental-solution} and Proposition \ref{prp:irred}, $\un{i}\in\mathsf{Irr}(\al_1,\al_2)$. By Remark \ref{rem:unique-Dyck-sequence} and Lemma \ref{lem:fundamental-Dyck-solutions}, $P^{\un{i}}(u-\mu)=P^{\pi,\la}$ for some unique $(\pi,\la)\in\mathscr{D}(m,n)\times\C$.

Conversely, let $(\pi,\la)\in\mathscr{D}(m,n)\times\C$ and 
suppose that $P^{\pi,\la}$ is reducible. Thus $P^{\pi,\la}=P'\cdot P''$ for some non-constant solutions $P', P''$. Multiplying the components together this implies that $P^{\pi,\la}_1(u)P^{\pi,\la}_2(u)=P'_1(u)P''_1(u)P'_2(u)P''_2(u)$. By Proposition \ref{prp:factorization} each of the two solutions $P'$ and $P''$ is divisible by a shifted fundamental solution. By Corollary \ref{cor:irreducible-fundamental-solution}, $P'_1(u)P'_2(u)$ and $P''_1(u)P''_2(u)$ each has at least $m+n$ roots. This contradicts the fact that the degree of $P^{\pi,\la}_1(u)P^{\pi,\la}_2(u)$  is equal to $m+n$.
\end{proof}

Recall that two solutions $(p_1,q_1)$ and $(p_2,q_2)$ to \eqref{eq:main-equation} are called equivalent if there exists a complex number $b$ such that $p_1(u)=p_2(u-b)$ and $q_1(u)=q_2(u-b)$. Combining Theorem~\ref{thm:main2} and \cite[Theorem~2.4]{Fukukawa2013pre} we get the following.

\begin{cor}
Let $\al$ and $\be$ be non-zero complex numbers. Assume that $m\al+n\be=0$ for some relatively prime positive integers $m$ and $n$. Then the number of irreducible solutions to $(p,q)$ to \eqref{eq:main-equation} up to equivalence equals $\frac{1}{m+n}\binom{m+n}{m}$.
\end{cor}

This completes the classification of \emph{irreducible} solutions in $S(\al_1,\al_2)$. However, factorization into irreducible solutions is not unique, as the following example shows.

\begin{example}\label{ex:non-unique-factorization}
Let $\al_1=2$, $\al_2=-3$.
For a sequence $\un{i}\in\mathsf{M}(\al_1,\al_2)$  we put
\[\eta(\un{i})=(\al_{i_1}, \al_{i_1}+\al_{i_2},\ldots,\al_{i_1}+\al_{i_2}+\cdots+\al_{i_N})\in\Z^N,\]
where $N=\ell(\un{i})$.
With
\begin{align*}
\un{i} =(11122),\quad&\Rightarrow\quad \eta(\un{i})=(2,4,6,3,0),\\
\un{j} =(12211),\quad&\Rightarrow\quad \eta(\un{j})=(2,-1,-4,-2,0),\\
\un{i}'=(12112),\quad&\Rightarrow\quad \eta(\un{i}')=(2,-1,1,3,0)\\
\un{j}'=\un{i}',\quad&\Rightarrow\quad \eta(\un{j}')=(2,-1,1,3,0).
\end{align*}
and
\[\la=\la'=0,\; \mu=\mu'=-\al_2=3\]
Consider
\begin{align*}
F=P^{\un{i}}(u-\la)  &=\big( (u-3)u,    \; (u-2)(u-4)(u-6) \big), \\ 
G=P^{\un{j}}(u-\mu)  &=\big( (u-2)(u+1),\; (u-5)(u-1)(u-3) \big), \\ 
H=P^{\un{i}'}(u-\la')&=\big( (u+1)u,    \; (u-2)(u-1)(u-3) \big), \\ 
K=P^{\un{j}'}(u-\mu')&=\big( (u-2)(u-3),\; (u-5)(u-4)(u-6) \big).
\end{align*}
Then $F,G,H,K\in S(2,-3)$ and $FG=HK$. By Theorem \ref{thm:main2}, $F,G,H,K$ are all irreducible. It is easy to see that $F,G\not\sim H,K$. In this sense, factorization into irreducible elements (even up to equivalence) in the monoid $S(\al_1,\al_2)$ is not unique.
\end{example}

To resolve this issue, as mentioned in the introduction, we will start by factoring a solution into $c$-integral solutions (see Definition \ref{dfn:integral-sol}).
Then, in the next section we define a natural partial order and require irreducible factorizations of $c$-integral solutions to be ordered.

\begin{lem}\label{lem:main3-lemma}
Let $\al$ and $\be$ be non-zero complex numbers.
For any solution $(p,q)$ to \eqref{eq:main-equation}, in which $p$ and $q$ are non-zero and monic, there exists a unique finite subset $C\subseteq\C/(\Z\al+\Z\be)$, and unique $c$-integral solutions $(p_c,q_c)\neq (1,1)$ for $c\in C$ such that
\begin{equation}
(p,q)=\prod_{c\in C} (p_c,q_c).
\end{equation}
\end{lem}
\begin{proof}
Let $\Ga=\Z\al+\Z\be$.
Let $Q=(Q_1,Q_2)\in S(\al,\be)$.
For $c\in \C/\Ga$, define $Q^c=(Q^c_1, Q^c_2)$ by
\[Q^c_i=\prod_{w\in Z(Q_i)\cap c} (u-w),\quad i=1,2,\]
where for a polynomial $P\in R$, $Z(P)$ is the multiset of zeros of $P$, with multiplicities.
Fix $c\in \C/\Ga$. Clearly $Q^c$ divides $Q$ in $R\times R$.
We claim that $Q^c\in S(\al,\be)$ for all $c\in \C/\Ga$. Write $Q=Q^c\cdot Q'$, where $Q'=(Q'_1,Q'_2)$.
Since $Q\in S(\al,\be)$,
\[Q_1(u)Q_2(u)=Q_1(u+\be)Q_2(u+\al)\]
hence
\[Q^c_1(u)Q^c_2(u)\cdot Q'_1(u)Q'_2(u)=Q^c_1(u+\be)Q^c_2(u+\al)\cdot Q'_1(u+\be)Q'_2(u+\al).\]
Let $u-w$ be any irreducible factor of $Q^c_1(u)Q^c_2(u)$. Thus $w\in c$. Then $u-w$ has to divide
$Q^c_1(u+\be)Q^c_2(u+\al)$ because those are the only factors in the right hand side with roots in $c$ (here we use that $c$ is a coset modulo $\Ga$). Conversely any irreducible factor of $Q^c_1(u+\be)Q^c_2(u+\al)$ has to be a factor of $Q^c_1(u)Q^c_2(u)$ in the left hand side. This shows that
\[Q^c_1(u)Q^c_2(u) = Q^c_1(u+\be)Q^c_2(u+\al)\]
i.e. $Q^c\in S(\al,\be)$.
Consequently we obtain
\[Q=\prod_{c\in\C/\Ga} Q^c\]
where $Q^c\in S(\al,\be)$ and every zero of $Q^c_1(u)Q^c_2(u)$ belongs to $c$ for all $c\in\C/\Ga$.
\end{proof}
\begin{rem}\label{rem:integral-factor}
 In particular, Lemma \ref{lem:main3-lemma} implies that if a solution $(p,q)=\prod_{c\in C}\left(p_c,q_c\right)$ is such that, for all $c\in C$ the solution $(p_c,q_c)$ is irreducible, then this factorization into irreducibles is unique.
\end{rem}

\section{Cylindrical Dyck paths and ordered factorizations}
\label{sec:ordered-factorizations}

Let $\al$ and $\be$ be non-zero complex numbers such that $m\al+n\be=0$ for a (unique) pair of relatively prime positive integers $m$ and $n$.
Let $\Ga=\Z\al+\Z\be$ and $c\in \C/\Ga$.

Let $Y=Y(m,n)$ be the discrete cylinder $Y=\Z^2/\Z(m,n)$. We will sometimes identify $Y$ with the quotient set $\iv{0}{m}\times\Z/\sim$ where $(m,y+n)\sim(0,y)$ for all $y\in\Z$.
For $(x,y)\in\Z^2$ we denote by $\overline{(x,y)}$ the image in $Y$ under the canonical map.

\begin{dfn} A \emph{cylindrical (generalized) Dyck path} $\overline{\pi}$ is a sequence of points on $Y$
\[\overline{\pi}=\overline{\pi}_{kl}=\left(\overline{(x_0+k,y_0+l)},\overline{(x_1+k,y_1+l)},\ldots,
\overline{(x_{m+n}+k,y_{m+n}+l)}\right)\]
where $\pi=\big((x_0,y_0),(x_1,y_1),\ldots,(x_{m+n},y_{m+n})\big)\in\mathscr{D}(m,n)$ and $(k,l)\in\iv{0}{m-1}\times\Z$. The point $(k,l)$ is called the \emph{base point} of $\overline{\pi}$. It is uniquely determined by $\overline{\pi}$ by Lemma \ref{lem:fuku}. The set of cylindrical Dyck paths is denoted by $\mathscr{D}^{\mathrm{cyl}}(m,n)$.
\end{dfn}

For fixed $z_0\in c$, which we refer to as a choice of \emph{origin}, there is a bijection
\begin{equation}\label{eq:cylindrical-bijection}
\begin{aligned}
\mathscr{D}(m,n)\times c&\longrightarrow \mathscr{D}^{\mathrm{cyl}}(m,n),\\
(\pi,z_0+k\al+l\be) &\longmapsto \overline{\pi}_{kl}.
\end{aligned}
\end{equation}

We now define a natural partial order on $\mathscr{D}^\mathrm{cyl}(m,n)$.
\begin{dfn}
Given $\overline{\pi},\overline{\tau}\in\mathscr{D}^{\mathrm{cyl}}(m,n)$ we define $\overline{\pi}\preceq\overline{\tau}$ if and only if $\overline{\pi}$ lies completely to the south-east of $\overline{\tau}$. That is, if and only if whenever $\overline{(x,y)}$ is a point on $\overline{\pi}$ and $\overline{(x',y')}$ is a point on $\overline{\tau}$ we have
\begin{equation}
\label{eq:partial-order-definition}
x+y=x'+y'\Longrightarrow y-x\le y'-x'.
\end{equation}
\end{dfn}

\begin{rem} Two cylindrical Dyck paths are comparable if and only if they are \emph{non-crossing}.
\end{rem}

We use the bijection \eqref{eq:cylindrical-bijection} to transfer the partial order to $\mathscr{D}(m,n)\times c$. The resulting order does not depend on the choice of $z_0$. Indeed, any other choice of origin $z_0'\in c$ is related to $z_0$ by $z_0'=z_0+r\al+s\be$ for some $r,s\in\Z$ since $c$ is a coset in $\C/\Ga$. Then the corresponding coordinates of points on cylindrical Dyck paths are all translated by a common vector $\overline{(r,s)}$ which does not change \eqref{eq:partial-order-definition}.

\begin{rem} \label{rem:choice-of-origin}
Given a finite subset $\{(\pi_1,\la_1),\ldots,(\pi_r,\la_r)\}\subset\mathscr{D}(m,n)\times c$ there is a natural choice of origin $z_0$. Namely $z_0=\min\{\la_1,\ldots,\la_r\}$ with respect to the total order on $c$ in which $\la\le\mu$ if and only if $\xi(\mu-\la)\ge 0$, where $\xi:\Ga\to\Z$ is the bijection given by $\xi(r\al+s\be)=r+sm$ where $(r,s)\in\iv{0}{m-1}\times\Z$. Then $\la_j\in z_0+\iv{0}{m-1}\al+\Z_{\ge 0}\be$ for all $j$.
\end{rem}

In the following examples we will use the coordinate map
\begin{equation}
\kappa=\kappa_{z_0}:c\to Y,\qquad \kappa(z_0+k\al+l\be)=\overline{(k,l)}.
\end{equation}

\begin{example}
Let $(\al,\be)=(-5,3)$, so that $(m,n)=(3,5)$ and we fix $z_0=0\in\Z\al+\Z\be$. Let $\pi_1=\pi(22212211)$, $\la_1=-10=2\al_1$, $\pi_2=\pi(22222111)$, $\la_2=-17=\al_1-4\al_2$, $\pi_3=\pi(22221211)$ and $\la_3=-16 =2\al_1-2\al_2$. We draw the cylindrical Dyck paths corresponding to $(\pi_1,\la_1)$ (red filled dots), $(\pi_2,\la_2)$ (blue squares), and $(\pi_3,\la_3)$ (green circles) below.

\begin{center}

\begin{tikzpicture}[xscale=1,yscale=0.8]
\draw[-] (0,6) to (0,-5);
\draw[dashed] (3,6) to (3,-5);
\draw (0,0) node[anchor=east] {$\overline{(0,0)}=\kappa(0)$};
\draw[-] (0,0) to (0,-2) [red];
\draw[-] (0,0) to (2,0) [red];
\draw[-] (2,0) to (2,3) [red];
\draw[-] (2,3) to (3,3) [red];

\draw[-] (0,-4) to (1,-4) [blue];
\draw[-] (1,-4) to (1,1) [blue];
\draw[-] (1,1) to (3,1) [blue];
\draw (1,-4) node[anchor=north] {$\kappa(\la_2)$};

\draw (2,0) node[anchor=west] {$\kappa(\la_1)$};

\draw (2,-2) node[anchor=west] {$\kappa(\la_3)$};
\draw (0,-3) to (0,-2) [OliveGreen];
\draw (0,-2) to (2,-2) [OliveGreen];
\draw (2,-2) to (2,0) [OliveGreen];
\draw (2.02,0) to (2.02,2) [OliveGreen];
\draw (2,2) to (3,2) [OliveGreen];

\fill[red] (0,0) circle(2pt);
\fill[red] (1,0) circle(2pt);
\fill[red] (2,0) circle(2pt);
\fill[red] (2,1) circle(2pt);
\fill[red] (2,2) circle(2pt);
\fill[red] (2,3) circle(2pt);
\fill[red] (3,3) circle(2pt);
\fill[red] (0,-1) circle(2pt);
\fill[red] (0,-2) circle(2pt);
\draw[OliveGreen] (0,-3) circle(2.5pt);
\draw[OliveGreen] (0,-2) circle(2.5pt);
\draw[OliveGreen] (1,-2) circle(2.5pt);
\draw[OliveGreen] (2,-2) circle(2.5pt);
\draw[OliveGreen] (2,-1) circle(2.5pt);
\draw[OliveGreen] (2,0) circle(2.5pt);
\draw[OliveGreen] (2,1) circle(2.5pt);
\draw[OliveGreen] (2,2) circle(2.5pt);
\draw[OliveGreen] (3,2) circle(2.5pt);
\draw[blue][xshift=-2.5pt,yshift=-2.5pt] (0,-4) rectangle++ (5pt,5pt);
\draw[blue][xshift=-2.5pt,yshift=-2.5pt] (1,-4) rectangle++ (5pt,5pt);
\draw[blue][xshift=-2.5pt,yshift=-2.5pt] (1,-3) rectangle++ (5pt,5pt);
\draw[blue][xshift=-2.5pt,yshift=-2.5pt] (1,-2) rectangle++ (5pt,5pt);
\draw[blue][xshift=-2.5pt,yshift=-2.5pt] (1,-1) rectangle++ (5pt,5pt);
\draw[blue][xshift=-2.5pt,yshift=-2.5pt] (1,0) rectangle++ (5pt,5pt);
\draw[blue][xshift=-2.5pt,yshift=-2.5pt] (1,1) rectangle++ (5pt,5pt);
\draw[blue][xshift=-2.5pt,yshift=-2.5pt] (2,1) rectangle++ (5pt,5pt);
\draw[blue][xshift=-2.5pt,yshift=-2.5pt] (3,1) rectangle++ (5pt,5pt);

\end{tikzpicture}
\end{center}

We can see in this case that $(\pi_3,\la_3)\preceq (\pi_1,\la_1)$ while $(\pi_2,\la_2)$ is not comparable to either of the other two paths.
If we made the natural choice of origin as in Remark \ref{rem:choice-of-origin}, then $z_0=-17$ and we would then redraw the picture as follows.

\begin{center}
\begin{tikzpicture}[xscale=1,yscale=0.8]
\draw[-] (0,11) to (0,0);
\draw[dashed] (3,11) to (3,0);
\draw (3,5) node[anchor=west] {$\overline{(3,5)}=\overline{(0,0)}$};
\draw (0,0) node[anchor=east] {$\overline{(0,0)}=\kappa(\la_2)$};
\draw[-] (0,4) to (1,4) [red];
\draw[-] (1,4) to (1,7) [red];
\draw[-] (1,7) to (2,7) [red];
\draw[-] (2,7) to (2,9) [red];
\draw[-] (2,9) to (3,9) [red];

\draw[-] (0,0) to (0,5) [blue];
\draw[-] (0,5) to (3,5) [blue];

\draw (1,4) node[anchor=west] {$\kappa(\la_1)$};

\draw (1,2) node[anchor=west] {$\kappa(\la_3)$};
\draw (1,2) to (1,4) [OliveGreen];
\draw (1,6) to (2,6) [OliveGreen];
\draw (2,6) to (2,7) [OliveGreen];
\draw (1.02,4) to (1.02,6) [OliveGreen];
\draw (2,7) to (3,7) [OliveGreen];
\draw (0,2) to (1,2) [OliveGreen];

\fill[red] (0,4) circle(2pt);
\fill[red] (1,4) circle(2pt);
\fill[red] (1,5) circle(2pt);
\fill[red] (1,6) circle(2pt);
\fill[red] (1,7) circle(2pt);
\fill[red] (2,7) circle(2pt);
\fill[red] (2,8) circle(2pt);
\fill[red] (2,9) circle(2pt);
\fill[red] (3,9) circle(2pt);
\draw[OliveGreen] (1,2) circle(2.5pt);
\draw[OliveGreen] (0,2) circle(2.5pt);
\draw[OliveGreen] (1,3) circle(2.5pt);
\draw[OliveGreen] (1,4) circle(2.5pt);
\draw[OliveGreen] (1,5) circle(2.5pt);
\draw[OliveGreen] (1,6) circle(2.5pt);
\draw[OliveGreen] (2,6) circle(2.5pt);
\draw[OliveGreen] (2,7) circle(2.5pt);
\draw[OliveGreen] (3,7) circle(2.5pt);
\draw[blue][xshift=-2.5pt,yshift=-2.5pt] (0,0) rectangle++ (5pt,5pt);
\draw[blue][xshift=-2.5pt,yshift=-2.5pt] (0,1) rectangle++ (5pt,5pt);
\draw[blue][xshift=-2.5pt,yshift=-2.5pt] (0,2) rectangle++ (5pt,5pt);
\draw[blue][xshift=-2.5pt,yshift=-2.5pt] (0,3) rectangle++ (5pt,5pt);
\draw[blue][xshift=-2.5pt,yshift=-2.5pt] (0,4) rectangle++ (5pt,5pt);
\draw[blue][xshift=-2.5pt,yshift=-2.5pt] (0,5) rectangle++ (5pt,5pt);
\draw[blue][xshift=-2.5pt,yshift=-2.5pt] (1,5) rectangle++ (5pt,5pt);
\draw[blue][xshift=-2.5pt,yshift=-2.5pt] (2,5) rectangle++ (5pt,5pt);
\draw[blue][xshift=-2.5pt,yshift=-2.5pt] (3,5) rectangle++ (5pt,5pt);

\end{tikzpicture}
\end{center}

Notice that the partial order of the three cylindrical Dyck paths has not been affected by the different choice of origin.
\end{example}

We can now prove the third and final main theorem from the introduction.

\begin{thm}\label{thm:main3}
Let $\al$ and $\be$ be non-zero complex numbers. Assume that $m\al+n\be=0$ for some relatively prime positive integers $m$ and $n$. Let $c\in \C/(\Z\al+\Z\be)$ and let $(p,q)$ be a $c$-integral solution to \eqref{eq:main-equation}. Then there exist a unique $k\in\Z_{\ge 0}$ and a unique sequence $\big((\pi_1,\la_1),(\pi_2,\la_2),\ldots,(\pi_k,\la_k)\big)$ of elements of $\mathscr{D}(m,n)\times c$ with
\begin{equation}
(\pi_1,\la_1)\preceq (\pi_2,\la_2)\preceq\cdots\preceq (\pi_k,\la_k)
\end{equation}
such that
\begin{equation}
(p,q)=
\left(P^{\pi_1,\la_1}_1,\, P^{\pi_1,\la_1}_2 \right)
\left(P^{\pi_2,\la_2}_1,\, P^{\pi_2,\la_2}_2 \right)
 \cdots
\left(P^{\pi_k,\la_k}_1,\, P^{\pi_k,\la_k}_2 \right),
\end{equation}
where the irreducible solutions $(P^{\pi,\la}_1,P^{\pi,\la}_2)$ for $(\pi,\la)\in \mathscr{D}(m,n)\times \C$ were defined in \eqref{eq:thm2-solutions}.
\end{thm}

\begin{proof}
Put $Q=(p,q)$. By Theorem \ref{thm:main2}, we have
$Q=P^{\tau_1,\mu_1}\cdots P^{\tau_k,\mu_k}$ for some $(\tau_i,\mu_i)\in\mathscr{D}(m,n)\times c$.
We will show that there exists a unique $k$-tuple
$\big((\pi_1,\la_1),\ldots,(\pi_k,\la_k)\big)$ of elements of $\mathscr{D}(m,n)\times c$ such that
\begin{equation}\label{eq:proof-main3-P1}
(\pi_1,\la_1)\preceq (\pi_2,\la_2)\preceq
\cdots \preceq (\pi_k,\la_k)
\end{equation}
and
\begin{equation} \label{eq:proof-main3-P2}
P^{\pi_1,\la_1}\cdots P^{\pi_r,\la_k}=
P^{\tau_1,\mu_1}\cdots P^{\tau_r,\mu_k}.
\end{equation}
We will use induction on $k$. For $k=1$,\eqref{eq:proof-main3-P2} implies that the only choice is $(\pi_1,\la_1)=(\tau_1,\mu_1)$ since $(\pi,\la)$ is uniquely determined by $P^{\pi,\la}$. Assume that $k>1$.
Then we claim that there exists a unique $(\pi_1,\la_1)\in\mathscr{D}(m,n)\times c$ with the properties
\begin{equation}\label{eq:proof-main3-pi1-minimal}
(\pi_1,\la_1)\preceq (\tau_i,\mu_i)\quad\text{for all $i\in\iv{1}{k}$,}
\end{equation}
and
\begin{equation}\label{eq:proof-main3-pi1-divides}
\text{$P^{\pi_1,\la_1}$ divides $Q$ in $S(\al,\be)$.}
\end{equation}
To this end, choose $z_0\in c$ to be the minimum of $\{\mu_1,\ldots,\mu_k\}$ as in Remark \ref{rem:choice-of-origin}. Let $\overline{\tau}_i$ be the image of $(\tau_i,\mu_i)$ under \eqref{eq:cylindrical-bijection}.
We define $(\pi_1,\la_1)$ by constructing its image $\overline{\pi_1}=\left(\overline{(x_0,y_0)},\ldots,\overline{(x_{m+n},y_{m+n})}\right)$ under \eqref{eq:cylindrical-bijection}. First put $\overline{(x_0,y_0)}=\overline{(0,0)}$. Then recursively, assume that we have defined $\overline{(x_i,y_i)}$ for some $i\in\iv{0}{m+n-1}$. If $\overline{(x_i+1,y_i)}$ is a point of at least one of the cylindrical Dyck paths $\overline{\tau}_j$ for $j\in\iv{1}{k}$, then we let $\overline{(x_{i+1},y_{i+1})}=\overline{(x_i+1,y_i)}$. Otherwise we let $\overline{(x_{i+1},y_{i+1})}=\overline{(x_i,y_i+1)}$.
Then \eqref{eq:proof-main3-pi1-minimal} is satisfied since at every step we chose to go right whenever possible, and \eqref{eq:proof-main3-pi1-divides} holds by comparing zeroes since every point of $\overline{\pi_1}$ is also a point of some $\overline{\tau_j}$. The uniqueness follows from the fact that \eqref{eq:proof-main3-pi1-minimal} forces us to choose right if possible, while \eqref{eq:proof-main3-pi1-divides} forces us to choose a point of some $\overline{\tau_j}$ at every step. This proves the claim about $(\pi_1,\la_1)$.

By Proposition \ref{prp:elementary-properties}~\eqref{part:divisibility-property} it follows that
$\tilde{Q}=Q/P^{\pi_1,\la_1}$ is also a solution with roots in $c$. By the induction hypothesis, it follows that $\tilde{Q}$ has a unique ordered factorization into irreducible shifted fundamental solutions: 
$\tilde{Q}=P^{\pi_2,\la_2}\cdots P^{\pi_k,\la_k}$,
where $(\pi_2,\la_2)\preceq\cdots\preceq(\pi_k,\la_k)$.
By \eqref{eq:proof-main3-pi1-minimal} we have $(\pi_1,\la_1)\preceq (\pi_j,\la_j)$ for $j\in\iv{2}{k}$. This finishes the proof.
\end{proof}

\begin{example}
Let $(\al,\be)=(2,-3)$, so that $(m,n)=(3,2)$ and we fix $z_0=0$. 
We first draw the cylindrical Dyck paths corresponding to $(\tau_1,\mu_1)$ (red filled dots) and $(\tau_2,\mu_2)$ (green circles), where $\tau_1=\pi(22111)$, $\mu_1=6=-2\al_2$, $\tau_2=\pi(22111)$, $\mu_2=-1=\al_1-\al_2$, 

\begin{center}
\begin{tikzpicture}[xscale=1,yscale=1]
\draw[-] (0,3) to (0,-3);
\draw[dashed] (3,3) to (3,-3);
\draw (0,0) node[anchor=east] {$\overline{(0,0)}=\kappa(0)$};
\fill (0,0) circle(2pt);
\fill (3,0) circle(2pt);
\draw (1,-1) node[anchor=west] {$\kappa(\mu_2)$};
\draw[-] (0,0) to (3,0) [red];
\draw[-] (0,0) to (0,-2) [red];
\draw (0,-2) node[anchor=east] {$\overline{(0,-2)}=\kappa(\mu_1)$};
\draw (0,-1) to (1,-1) [OliveGreen];
\draw (1,-1) to (1,1) [OliveGreen];
\draw (1,1) to (3,1) [OliveGreen];
\fill[red] (0,-1) circle(2pt);
\fill[red] (0,-2) circle(2pt);
\fill[red] (0,0) circle(2pt);
\fill[red] (1,0) circle(2pt);
\fill[red] (2,0) circle(2pt);
\fill[red] (3,0) circle(2pt);
\draw[OliveGreen] (0,-1) circle(3pt);
\draw[OliveGreen] (1,-1) circle(3pt);
\draw[OliveGreen] (1,0) circle(3pt);
\draw[OliveGreen] (1,1) circle(3pt);
\draw[OliveGreen] (2,1) circle(3pt);
\draw[OliveGreen] (3,1) circle(3pt);

\draw (3,0) node[anchor=west] {$\overline{(m,n-2)}= \overline{(0,-2)}$};
\end{tikzpicture}
\end{center}

In this second diagram we have the cylindrical Dyck paths corresponding to $(\pi_1,\la_1)$ (red filled dots) and $(\pi_2,\la_2)$ (green circles) with $\pi_1=\pi(21211)$, $\la_1=3=-\al_2$, $\pi_2=\pi(21211)$, $\la_2=6=-2\al_2$.

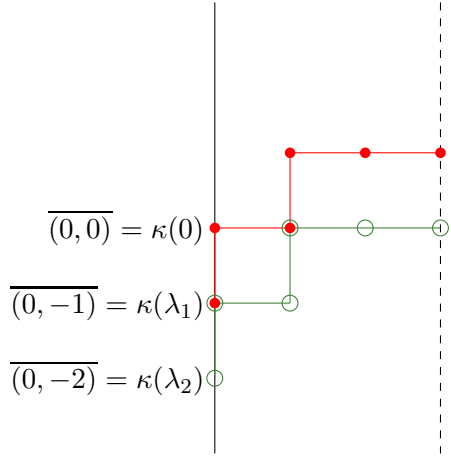
\begin{center}
\begin{tikzpicture}[xscale=1,yscale=1]
\draw[-] (0,3) to (0,-3);
\draw[dashed] (3,3) to (3,-3);
\draw (0,0) node[anchor=east] {$\overline{(0,0)}=\kappa(0)$};
\draw (0,-1) node[anchor=east] {$\overline{(0,-1)}=\kappa(\la_1)$};
\draw (0,-2) node[anchor=east] {$\overline{(0,-2)}=\kappa(\la_2)$};
\draw[-] (0,-1) to (0,0) [red];
\draw[-] (0,0) to (1,0) [red];
\draw[-] (1,0) to (1,1) [red];
\draw[-] (1,1) to (3,1) [red];

\draw (0,-2) to (0,-1) [OliveGreen];
\draw (0,-1) to (1,-1) [OliveGreen];
\draw (1,-1) to (1,0) [OliveGreen];
\draw (1,0) to (3,0) [OliveGreen];
\fill[red] (0,-1) circle(2pt);
\fill[red] (0,0) circle(2pt);
\fill[red] (1,0) circle(2pt);
\fill[red] (1,1) circle(2pt);
\fill[red] (2,1) circle(2pt);
\fill[red] (3,1) circle(2pt);
\draw[OliveGreen] (0,-2) circle(3pt);
\draw[OliveGreen] (0,-1) circle(3pt);
\draw[OliveGreen] (1,-1) circle(3pt);
\draw[OliveGreen] (1,0) circle(3pt);
\draw[OliveGreen] (2,0) circle(3pt);
\draw[OliveGreen] (3,0) circle(3pt);

\end{tikzpicture}
\end{center}

Note that $(\pi_2,\la_2)\preceq (\pi_1,\la_1)$, while $(\tau_1,\mu_1)$ and $(\tau_2,\mu_2)$ are not comparable. In addition, these pictures represent the solutions of Example \ref{ex:non-unique-factorization}, in fact $F=P^{\tau_1,\mu_1}$, $G=P^{\tau_2,\mu_2}$, $H=P^{\pi_1,\la_1}$, and $K=P^{\pi_2,\la_2}$.
Thus the unique ordered factorization into irreducibles for the solution of Example \ref{ex:non-unique-factorization} is $KH$.
\end{example}

\section{Relation to previous solutions}
\label{sec:previous-solutions}

In \cite{HarSer2015pre} a family of solutions to the consistency equations \eqref{eq:TGWA-consistency} were attached to any multiquiver. The case when the quiver has only two vertices, connected by a single edge, corresponds exactly to the univariate rank two case considered in the present paper. In this section we will give the factorization into irreducibles for those solutions. It turns out that all irreducible factors are associated to certain special generalized Dyck paths.

For positive integers $m$ and $n$, recall that $\mathscr{D}(m,n)$ denotes the set of all generalized Dyck paths in $\Z^2$ from $(0,0)$ to $(m,n)$.
The \emph{area} of a generalized Dyck path $\pi\in \mathscr{D}(m,n)$,   denoted $\mathrm{Area}(\pi)$, is defined as the number of complete lattice cells contained between $\pi$ and the line $y=(n/m)x$.

\begin{lem}
Let $\al_1$ and $\al_2$ be two non-zero integers of opposite sign and put $N=|\al_1|+|\al_2|$.
\begin{enumerate}[{\rm (a)}]
\item There exists a unique sequence $\un{i}_0=(i_1i_2\cdots i_N)\in\mathsf{M}(\al_1,\al_2)$ with the property that
\begin{equation}\label{eq:zero-area-condition}
\al_{i_1}+\al_{i_2}+\cdots+\al_{i_k}\in\iv{0}{N-1}\quad\forall k\in\iv{1}{N}.
\end{equation}
\item The corresponding lattice path $\pi_0=\pi(\un{i}_0)$ is the unique Dyck path in $\mathscr{D}(|\al_2|,|\al_1|)$ with zero area.
\end{enumerate}
\end{lem}
\begin{proof}
We will assume $\al_1<0$ and $\al_2>0$, the other case being symmetric.

(a) For $k=1$ we clearly must choose $i_1=2$. For $k>1$ note that we have a disjoint decomposition
\[\iv{0}{N-1}=\iv{0}{-\al_1-1}\cup \iv{-\al_1}{N-1}.\]
If $\al_{i_1}+\cdots+\al_{i_{k-1}}$ belongs to the first interval, adding $\al_1$ would bring us outside of the interval $\iv{0}{N-1}$. Similarly, if it belongs to the second interval we cannot add $\al_2$ because $\al_2-\al_1=N>N-1$. Therefore the only possibility is to define
\[i_{k}=\begin{cases}
1,& \text{if $\al_{i_1}+\cdots+\al_{i_{k-1}}\in \iv{-\al_1}{N-1}$}, \\ 
2,& \text{if $\al_{1_1}+\cdots+\al_{i_{k-1}}\in \iv{0}{-\al_1-1}$}.
\end{cases}\]
With this choice, the condition $\al_{i_1}+\cdots+\al_{i_k}\in\iv{0}{N-1}$ is indeed satisfied.

(b) Let $\pi=\pi(\un{i})$ be the lattice path corresponding to a sequence $\un{i}\in\mathsf{M}(\al_1,\al_2)$, starting at $(0,0)$ and ending at $(\al_2,-\al_1)$. Condition \eqref{eq:zero-area-condition} is equivalent to that
\begin{equation}\label{eq:zero-area-inequalities}
 0\le x\al_1+y\al_2 \le N-1
\end{equation}
for all lattice points $(x,y)$ belonging to $\pi$.
Since $N=\al_2-\al_1$, the inequalities \eqref{eq:zero-area-inequalities} are equivalent to that $(x,y)$ lies above and $(x+1,y-1)$ lies below the line $y=(-\al_1/\al_2)x$. This in turn means precisely that $\pi$ is the unique Dyck path with zero area. 
\end{proof}

In \cite{HarSer2015pre}, to any matrix satisfying some conditions (being the incidence matrix of a multiquiver), a ring $R$, certain automorphisms $\si_i$ and elements $t_i\in R$ were constructed such that the consistency equations \eqref{eq:TGWA-consistency} hold. If we consider the case of an $1\times 2$ matrix $[\al_1\;\al_2]$, the condition is that $\al_1$ and $\al_2$ are non-zero integers of opposite sign and we obtain the following: $R=\C[u]$, $\si_i(u)=u-\al_i$, and for $i=1,2$,
\begin{equation}\label{eq:HS-solutions}
t_i=\begin{cases}
u(u+1)\cdots (u+\al_i-1),& \al_i>0,\\
(u-1)(u-2)\cdots(u-|\al_i|),& \al_i<0.
\end{cases}
\end{equation}
By \cite[Theorem~2.4]{HarSer2015pre}, the pair of polynomials $(t_1,t_2)$ is a solution to \eqref{eq:main-equation}.
The following proposition shows that the factorization of these solutions into irreducibles is given by fundamental solutions associated to generalized Dyck paths of zero area.
Below we put $P^\pi=P^{\un{i}(\pi)}=P^{\pi,0}$ for $\pi\in\mathscr{D}(m,n)$.
\begin{prp}
Let $\al_1$ and $\al_2$ be two non-zero integers of opposite sign. Let $\ga$ be the positive greatest common divisor of $\al_1$ and $\al_2$ and put $\al=\max\{\al_1,\al_2\}$. 
Then we have
\begin{equation}\label{eq:multigraph-factorization} 
(t_1,t_2)=
P^{\pi_0}(u+\al-1)P^{\pi_0}(u+\al-2)\cdots P^{\pi_0}(u+\al-\ga)
\end{equation} 
where $\pi_0$ is the unique zero area generalized Dyck path from $(0,0)$ to $(m,n)$ where $m=|\al_2/\ga|$ and $n=|\al_1/\ga|$.
\end{prp}
By Remark \ref{rem:integral-factor}, the factorization into irreducibles \eqref{eq:multigraph-factorization} is unique, since all factors have roots in different cosets modulo $\Z\al_1+\Z\al_2=\Z\ga$.
Rather than giving a detailed proof, we provide some examples.
\begin{example}
(Type $A_2$) Let $\al_1=-1$ and $\al_1=1$.
Then $(t_1,t_2)=P^{\pi_0}(u)=\big(u-1,\,u\big)$.
\end{example}
\begin{example}
(Type $C_2$) Let $\al_1=-1$ and $\al_1=2$.
Then $(t_1,t_2)=P^{\pi_0}(u+1)=\big(u-1,\,u(u+1)\big)$.
\end{example}
\begin{example}
Let $\al_1=-4$ and $\al_2=6$. Then $\ga=2$, $m=3$, $n=2$, $\un{i}_0=(21211)$ and
$\pi_0=\big((0,0),(0,1),(1,1),(1,2),(2,2),(3,2)\big)$.
We have
$P^{\pi_0}(u)=\big((u-6)(u-8),\,(u-2)(u-4)u\big)$
and hence
\begin{align*}
(t_1,t_2)&=\big((u-1)(u-2)(u-3)(u-4),\, u(u+1)(u+2)(u+3)(u+4)(u+5)\big)\\
&=P^{\pi_0}(u+5)P^{\pi_0}(u+4).
\end{align*}
\end{example}

\bibliographystyle{siam}

\end{document}